\documentclass{article}
\usepackage{amscd,amsmath,amssymb,amsthm,enumerate,epsfig,graphicx,mathrsfs,tikz,wasysym}

\newtheorem{theorem}{Theorem}[section]
\newtheorem{lemma}[theorem]{Lemma}
\newtheorem{proposition}[theorem]{Proposition}

\newtheorem{remark}[theorem]{Remark}
\theoremstyle{definition}
\newtheorem{definition}[theorem]{Definition}

\title{Branching Laws of Generalized Verma Modules for Non-symmetric Polar Pairs}
\author{Haian HE}
\date{}

\begin{document}

\maketitle

\begin{abstract}
We give branching formulas from $so(7,\mathbb{C})$ to $\mathfrak{g}_2$ for generalized Verma modules attached to $\mathfrak{g}_2$-compatible parabolic subalgebras of $so(7,\mathbb{C})$, and branching formulas from $\mathfrak{g}_2$ to $sl(3,\mathbb{C})$ for generalized Verma modules attached to $sl(3,\mathbb{C})$-compatible parabolic subalgebras of $\mathfrak{g}_2$ respectively, under some assumptions on the parameters of generalized Verma modules.
\end{abstract}

\section{Introduction}

Branching law is one of the fundamental topics in representation theory. It gives multiplicities for a representation (or module) of a group (or an algebra) upon restriction to its subgroup (or subalgebra). In this paper, we partially study the branching problems for the Lie algebra pairs $(so(7,\mathbb{C}),\mathfrak{g}_2)$ and $(\mathfrak{g}_2,sl(3,\mathbb{C}))$.

The Lie algebra pairs $(so(7,\mathbb{C}),\mathfrak{g}_2)$ and $(\mathfrak{g}_2,sl(3,\mathbb{C}))$ are studied by many mathematicians, especially for $(so(7,\mathbb{C}),\mathfrak{g}_2)$. The following two points indicate that this two pairs are special.

Firstly, although $(so(7,\mathbb{C}),\mathfrak{g}_2)$ and $(\mathfrak{g}_2,sl(3,\mathbb{C}))$ are not symmetric pairs, they behave like symmetric pairs. Concretely, let's consider the simply connected compact real Lie group pairs $(Spin(7),G_2)$ and $(G_2,SU(3))$. In [\textbf{HPTT}], the authors gave the definition of polar pair (Definition 3.1 [\textbf{HPTT}]) and irreducible polar pair (Definition 3.2 [\textbf{HPTT}]), and did a classification of irreducible polar pairs $(G,H)$. If $(G,H)$ is an irreducible polar pair with $G$ semi-simple and simply connected and $H$ connected, the $(G,H)$ is either a symmetric pair associated to some symmetric space of compact type or else it is isomorphic to $(Spin(7),G_2)$ or $(G_2,SU(3))$ (Theorem 3.12 [\textbf{HPTT}]). One can check that both $Spin(7)/G_2\cong S^7$ and $G_2/SU(3)\cong S^6$ are symmetric spaces, but neither of $(Spin(7),G_2)$ and $(G_2,SU(3))$ is a symmetric pair. The polar pairs of compact semi-simple Lie algebras were discussed in [\textbf{H}] where J.S.Huang studied harmonic analysis on the compact polar homogeneous spaces. Besides symmetric pairs, there are only two compact polar pairs of Lie algebras: $(so(7),\mathfrak{g}_2)$ and $(\mathfrak{g}_2,su(3))$. The complexified types of them are $(so(7,\mathbb{C}),\mathfrak{g}_2)$ and $(\mathfrak{g}_2,sl(3,\mathbb{C}))$, which we call non-symmetric polar pairs without confusion.

Secondly, there exists a class of pairs which are called weakly symmetric pairs. The definition of weakly symmetric pair can be found in Definition 12.2.3 [\textbf{W}]. It can be showed that all the symmetric pairs are weakly symmetric pairs, and all the weakly symmetric pairs are Gelfand pairs. The complex reductive non-symmetric weakly symmetric pairs $(G,H)$ with $G$ simple were completely listed in Table 12.7.2 [\textbf{W}]. The twelve pairs can be divided into three classes. The first class consists of $(G,H)$ such that there exists a symmetric subgroup $K$ of $G$ satisfying $H=[K,K]$, the commutator group of $H$. In this class, branching law for each $(G,H)$ is close to that for the symmetric pair $(G,K)$ to some extent. The second class consists of $(G,H)$ such that there exists a symmetric subgroup $K$ of $G$, and $(K/N,H/N)$ forms a symmetric pair as well where $N$ is a normal subgroup of $K$ and $H$. In this case, branching law for $(G,H)$ can be regarded as two steps: from $G$ to $K$, and from $K$ to $H$, both of which are related to branching laws for symmetric pairs. The remaining three pairs are included in the third class. They are $(Spin(8,\mathbb{C}),G_2)$, $(Spin(7,\mathbb{C}),G_2)$ and $(G_2,SL(3,\mathbb{C}))$. As to $(Spin(8,\mathbb{C}),G_2)$, the branching law can be processed in two stages: $(Spin(8,\mathbb{C}),Spin(7,\mathbb{C}))$ and $(Spin(7,\mathbb{C}),G_2)$, where the former one is a symmetric pair. Hence, the only two pairs whose branching laws seem not related to symmetric pairs among weakly symmetric pairs are $(Spin(7,\mathbb{C}),G_2)$ and $(G_2,SL(3,\mathbb{C}))$.

T.Kobayashi studied the branching problems for symmetric pairs (Section 4 [\textbf{Ko}]). For a complex semi-simple symmetric pair $(\mathfrak{g},\mathfrak{g}')$, when a parabolic subalgebra $\mathfrak{p}$ of $\mathfrak{g}$ is $\mathfrak{g}'$-compatible and has abelian nilpotent radical, then a generalized Verma module associated to $\mathfrak{p}$ with a sufficient negative highest weight decomposes as multiplicity-free direct sum of generalized Verma modules of $\mathfrak{g}'$. Moreover, he gave the branching formulas of Verma modules for three classical symmetric pairs. We shall introduce his key methods in details later. In this paper, we shan't discuss a more general result of branching formulas for symmetric pairs; instead, we shall focus on the two non-symmetric polar pairs $(so(7,\mathbb{C}),\mathfrak{g}_2)$ and $(\mathfrak{g}_2,sl(3,\mathbb{C}))$.

In $[\textbf{MS1}]$, T.Milev and P.Somberg studied the branching law of generalized Verma modules for $(so(7,\mathbb{C}),\mathfrak{g}_2)$, and gave the lists of $\bar{\mathfrak{b}}$-singular vectors. Actually, they are just the highest weight vectors of generalized Verma modules as $\mathfrak{g}_2$-modules. Moreover, an useful method called F-method was introduced in $[\textbf{MS2}]$ in order to find $\tilde{L}'$-singular vectors. All these vectors give much information for branching formulas. We shall see that these results will be contained in our branching formulas below.

W.M.McGovern gave an explicit branching formula for $(so(7,\mathbb{C}),\mathfrak{g}_2)$ of finite dimensional modules (Theorem 3.4 [\textbf{M}]). In fact, branching formulas of finite dimensional modules can always be deduced by Kostant's Branching Theorem (Theorem 8.2.1 [\textbf{GW}] or Theorem 9.20 [\textbf{K}]). But we know little about those for infinite dimensional modules. We are not able to deal with all of infinite dimensional modules; instead, we shall only handle generalized Verma modules. It is well known that generalized Verma modules belong to the generalized BGG category $\mathcal{O}^\mathfrak{p}$. Here, we require $\mathfrak{p}$ to be a $\mathfrak{g}_2$(respectively, $sl(3,\mathbb{C})$)-compatible parabolic subalgebra of $so(7,\mathbb{C})$(respectively, $\mathfrak{g}_2$) (Definition 3.7 [\textbf{Ko}]), the reason for which we shall explain in Section 3. Moreover, we have to require the parameter of each parabolic Verma module to be ``generic'' so that the generalized Verma module is simple. Under these requirements, any sub-quotient occurring in the restriction of each generalized Verma module lies in $\mathcal{O}^{\mathfrak{p}'}$ where $\mathfrak{p}'$ is the intersection of $\mathfrak{p}$ with $\mathfrak{g}_2$(respectively, $sl(3,\mathbb{C})$) (Lemma 3.4 and Proposition 3.8 [\textbf{Ko}]).

As to the basic structures of the complex Lie algebras of $so(7,\mathbb{C})$ and $sl(3,\mathbb{C})$, it is easy to understand them because both of them are of classical types. For the exceptional Lie algebra $\mathfrak{g}_2$, its structure is showed in detail in Chapter 22 [\textbf{FH}] and Section 19.3 [\textbf{Hu1}]. T.Levasseur and S.P.Smith described the details of the inclusion $\mathfrak{g}_2\subseteq so(7,\mathbb{C})$ (Section 2 [\textbf{LS}]). We shall restate the embedding in Section 2 and use this construction throughout this paper.

In Section 3, we shall first briefly introduce the concept of discretely decomposable representation given by T.Kobayashi, which will explain why we require compatibility of parabolic subalgebras. Then we shall find all $\mathfrak{g}_2$-compatible standard parabolic subalgebras of $so(7,\mathbb{C})$. Our main part will begin from Section 4. In Section 4, We shall first recall T.Kobayashi's method which will indicate a way for decomposition in Grothendieck group level, and making use of which we shall compute the decompositions of generalized Verma modules attached to compatible parabolic subalgebras in the Grothendieck groups of $\mathcal{O}^{\mathfrak{p}'}$ for $(so(7,\mathbb{C}),\mathfrak{g}_2)$. In Section 5, we'll show that under some special assumptions on the parameters, the decompositions in the Grothendieck groups of $\mathcal{O}^{\mathfrak{p}'}$ given in Section 4 are just $\mathfrak{g}_2$-module decompositions. For the pair $(\mathfrak{g}_2,sl(3,\mathbb{C}))$, the method is parallel but the computation is much easier than that for $(so(7,\mathbb{C}),\mathfrak{g}_2)$. We shall only give the parallel results for it without computation in Section 6.

Throughout the paper, we shall use the following notations.

Let $\mathbb{N}$, $\mathbb{Z}^+$, $\mathbb{Z}$ and $\mathbb{C}$ denote the set of nonnegative integers, positive integers, integers and complex numbers respectively. Let $\mathfrak{g}$ be a complex reductive Lie algebra. Then denote $\mathfrak{h}_\mathfrak{g}$ to be a Cartan subalgebra of $\mathfrak{g}$ with its dual space $\mathfrak{h}_\mathfrak{g}^*$, and denote $\Phi(\mathfrak{g})$, $\Phi^+(\mathfrak{g})$, $\Delta(\mathfrak{g})$ and $\Lambda^+(\mathfrak{g})$ to be root system, positive root system, simple root system and dominant integral weight system of $\mathfrak{g}$ respectively. If $\alpha$ is a root, let $H_\alpha$ be the corresponding co-root. And let $\rho(\mathfrak{g})$ denote half the sum of $\Phi^+(\mathfrak{g})$. We denote $W_\mathfrak{g}$ to be the Weyl group of $\mathfrak{g}$ generated by the reflections $s_\alpha$ for $\alpha\in\Phi(\mathfrak{g})$. We always denote $\mathfrak{p}$ to be a parabolic subalgebra of $\mathfrak{g}$, and let $\mathfrak{p}=\mathfrak{l}+\mathfrak{u}_+$ be a Levi decomposition with $\mathfrak{l}$ reductive subalgebra and $\mathfrak{u}_+$ nilpotent radical. Because a standard parabolic subalgebra is determined by a subset of the simple system and a Borel subalgebra which it contains (Lemma 3.8.1(ii) [\textbf{CM}] or Proposition 5.90 [\textbf{K}]), we denote $\mathfrak{p}_\Pi$ to be the standard parabolic subalgebra corresponding to the subset $\Pi\subseteq\Delta(\mathfrak{g})$; namely, $\mathfrak{p}_\Pi=\mathfrak{h}_\mathfrak{g}+\displaystyle{\sum_{\alpha\in\Phi^+(\mathfrak{g})}}\mathfrak{g}_\alpha+ \displaystyle{\sum_{\alpha\in\mathbb{N}\Pi\cap\Phi^+(\mathfrak{g})}}\mathfrak{g}_{-\alpha}$ where $\mathfrak{g}_\alpha$ is the root space of the root $\alpha$. An extreme case of parabolic subalgebra is Borel subalgebra which we denote by $\mathfrak{b}_\mathfrak{g}$. Moreover, denote $U(\mathfrak{g})$ to be the universal enveloping algebra of $\mathfrak{g}$. If $S$ is a set, we denote  $\mathrm{Card}S$ to be the cardinality of $S$.

\section{Embedding of $\mathfrak{g}_2$ in $so(7,\mathbb{C})$}

We briefly recall the embedding of $\mathfrak{g}_2$ in $so(7,\mathbb{C})$ described by T.Levasseur and S.P.Smith (Section 5 [\textbf{LS}]) in this section.

We may realize $so(7,\mathbb{C})$ as \[\{\left(\begin{array}{cccc}A&B&u\\C&-A^t&v\\-v^t&-u^t&0\end{array}\right)\mid u,v\in\mathbb{C}^3,A\in gl(3,\mathbb{C}),B,D skew-symmetric\},\] the set of skew-adjoint matrices relative to the quadratic form $2(z_1z_4+z_2z_5+z_3z_6)+z_7^2$ on $\mathbb{C}^7$. Let $E_{ij}\in gl(7,\mathbb{C})$ for $1\leq i,j\leq7$ be matrix such that $(i,j)$-entry is 1 and other entries are all 0. Define a Cartan subalgebra $\mathfrak{h}_{so(7,\mathbb{C})}$ of $so(7,\mathbb{C})$ with basis $\{H_i=E_{ii}-E_{i+3,i+3}\mid1\leq i\leq3\}$. Take a dual basis to $H_i$ in $\mathfrak{h}_{so(7,\mathbb{C})}^*$, $\{\varepsilon_i\mid1\leq i\leq3\}$. A root system of $so(7,\mathbb{C})$ is given by $\Phi(so(7,\mathbb{C}))=\{\pm\varepsilon_i\pm\varepsilon_j\mid1\leq i<j\leq3\}\cup\{\pm\varepsilon_k\mid1\leq k\leq3\}$ and a system of simple roots given by $\Delta(so(7,\mathbb{C}))=\{\varepsilon_1-\varepsilon_2,\varepsilon_2-\varepsilon_3,\varepsilon_3\}$.

The subalgebra $\mathfrak{g}_2$ is given by the Chevalley basis in terms of that of $so(7,\mathbb{C})$ below.
\begin{eqnarray*}
\begin{array}{rclcrcl}
X_{\alpha_2}&=&X_{\varepsilon_2-\varepsilon_3},\\
X_{3\alpha_1+\alpha_2}&=&-X_{\varepsilon_1+\varepsilon_3},\\
X_{3\alpha_1+2\alpha_2}&=&-X_{\varepsilon_1+\varepsilon_2},\\
X_{\alpha_1}&=&X_{\varepsilon_3}+X_{\varepsilon_1-\varepsilon_2},\\
X_{\alpha_1+\alpha_2}&=&X_{\varepsilon_2}-X_{\varepsilon_1-\varepsilon_3},\\
X_{2\alpha_1+\alpha_2}&=&-X_{\varepsilon_1}-X_{\varepsilon_2+\varepsilon_3},\\
X_{-\alpha_2}&=&X_{\varepsilon_3-\varepsilon_2},\\
X_{-3\alpha_1-\alpha_2}&=&X_{-\varepsilon_1-\varepsilon_3},\\
X_{-3\alpha_1-2\alpha_2}&=&X_{-\varepsilon_1-\varepsilon_2},\\
X_{-\alpha_1}&=&X_{-\varepsilon_3}+X_{\varepsilon_2-\varepsilon_1},\\
X_{-\alpha_1-\alpha_2}&=&X_{-\varepsilon_2}-X_{\varepsilon_3-\varepsilon_1},\\
X_{-2\alpha_1-\alpha_2}&=&-X_{-\varepsilon_1}-X_{-\varepsilon_2-\varepsilon_3},\\
H_{\alpha_1}&=&H_{\varepsilon_1-\varepsilon_2}+H_{\varepsilon_3}=H_1-H_2+2H_3,\\
H_{\alpha_2}&=&H_{\varepsilon_2-\varepsilon_3}=H_2-H_3.
\end{array}
\end{eqnarray*}

A Cartan subalgebra of $\mathfrak{g}_2$ is complex linearly spanned by $\{H_{\alpha_1},H_{\alpha_2}\}$. The inclusion $\mathfrak{h}_{\mathfrak{g}_2}\subseteq\mathfrak{h}_{so(7,\mathbb{C})}$ induces a restriction map $\mathrm{Res}_{\mathfrak{g}_2}^{so(7,\mathbb{C})}:\mathfrak{h}_{so(7,\mathbb{C})}^*\to\mathfrak{h}_{\mathfrak{g}_2}^*$. Fix simple system $\Delta(\mathfrak{g}_2)=\{\alpha_1,\alpha_2\}$ of $\mathfrak{g}_2$ with $\alpha_1=\mathrm{Res}_{\mathfrak{g}_2}^{so(7,\mathbb{C})}(\varepsilon_1-\varepsilon_2)=\mathrm{Res}_{\mathfrak{g_2}}^{so(7,\mathbb{C})}(\varepsilon_3)$ and $\alpha_2=\mathrm{Res}_{\mathfrak{g}_2}^{so(7,\mathbb{C})}(\varepsilon_2-\varepsilon_3)$.

We may write $so(7,\mathbb{C})=\mathfrak{g}_2\oplus U$, where $U$ is the orthogonal complement to $\mathfrak{g}_2$ with respect to Killing form of $so(7,\mathbb{C})$. Then $\dim U=7$, and since $[\mathfrak{g}_2, U]\neq0$, the only possibility is that $U$ is isomorphic to the unique 7-dimensional simple module of $\mathfrak{g}_2$. It is well known that $U=\mathrm{Span}_\mathbb{C}\{v_{\pm(2\alpha_1+\alpha_2)}, v_{\pm(\alpha_1+\alpha_2)}, v_{\pm\alpha_1}, v_0\}$, where $v_\beta$ for $\beta\in\Phi(\mathfrak{g}_2)$ are weight vectors in the corresponding weight spaces and $v_0$ is given by $H_{\varepsilon_2+\varepsilon_3}-\frac{1}{2}H_{\varepsilon_1}=H_2+H_3$, which satisfy the following relations:
\begin{eqnarray*}
\begin{array}{rclcrcl}
v_{\alpha_1+\alpha_2}&=&X_{-\alpha_1}v_{2\alpha_1+\alpha_2},\\
v_{\alpha_1}&=&X_{-\alpha_2}v_{\alpha_1+\alpha_2},\\
v_0&=&X_{-\alpha_1}v_{\alpha_1},\\
v_{-\alpha_1}&=&X_{-\alpha_1}v_0,\\
v_{-\alpha_1-\alpha_2}&=&X_{-\alpha_2}v_{-\alpha_1},\\
v_{-2\alpha_1-\alpha_2}&=&X_{-\alpha_1}v_{-2\alpha_1-\alpha_2}.
\end{array}
\end{eqnarray*}

\section{Compatible Parabolic Subalgebra}

In this section, our aim is to find all $\mathfrak{g}_2$-compatible standard parabolic subalgebras of $so(7,\mathbb{C})$. We need to give a general definition of compatible parabolic subalgebra. However, we should explain why we require $\mathfrak{p}$ to be a compatible parabolic subalgebra. Hence, we introduce T.Kobayashi's work at first, which will answer this question.

\begin{center}
\textbf{3.1 Discretely decomposable branching laws}
\end{center}

Suppose that $\mathfrak{g}$ is a complex reductive Lie algebra.
\begin{definition}
We say a $\mathfrak{g}$-module $X$ is $\emph{discretely decomposable}$ if there is an increasing filtration $\{X_m\}$ of $\mathfrak{g}$-submodules of finite length such that $X=\bigcup_{m=0}^{+\infty}X_m$. Further, we say $X$ is discretely decomposable in the category $\mathcal{O}^\mathfrak{p}$ if all $X_m$ can be taken from $\mathcal{O}^\mathfrak{p}$.
\end{definition}
Suppose that $\mathfrak{g}'\subseteq\mathfrak{g}$ is a reductive subalgebra, and $\mathfrak{p}'$ its parabolic subalgebra.
\begin{lemma}(Lemma 3.4 [\textbf{Ko}])
Let $X$ be a simple $g$-module. Then the restriction $\mathrm{Res}_{\mathfrak{g}'}^\mathfrak{g}X$ is discretely decomposable in the category $\mathcal{O}^{\mathfrak{p}'}$ if and only if there exists a $\mathfrak{g}'$-module $Y\in\mathcal{O}^{\mathfrak{p}'}$ such that $\hom_{\mathfrak{g}'}(Y,\mathrm{Res}_{\mathfrak{g}'}^\mathfrak{g}X)\neq\{0\}$. In this case, any sub-quotient occurring in the $\mathfrak{g}'$-module $\mathrm{Res}_{\mathfrak{g}'}^\mathfrak{g}X$ lies in $\mathcal{O}^{\mathfrak{p}'}$.
\end{lemma}
Let $G=\mathrm{Int}(\mathfrak{g})$, $P$ the parabolic subgroup of $G$ with Lie algebra $\mathfrak{p}$ as before, and $G'$ a reductive subgroup with Lie algebra $\mathfrak{g}'$.
\begin{proposition}(Proposition 3.5 [\textbf{Ko}])
If $G'P$ is closed in $G$ then the restriction $\mathrm{Res}_{\mathfrak{g}'}^\mathfrak{g}X$ is discretely decomposable for any simple $\mathfrak{g}$-module $X$ in $\mathcal{O}^\mathfrak{p}$.
\end{proposition}
A semi-simple element $H\in\mathfrak{g}$ is said to be $\emph{hyperbolic}$ if the eigenvalues of $\mathrm{ad}(H)$ are all real. For a hyperbolic element $H$, we define the subalgebras
\begin{center}
$\mathfrak{u}_+\equiv\mathfrak{u}_+(H)$, $\mathfrak{l}\equiv\mathfrak{l}(H)$, $\mathfrak{u}_-\equiv\mathfrak{u}_-(H)$
\end{center}
as the sum of the eigenspaces with positive, zero, and negative eigenvalues, respectively. Then
\begin{center}
$\mathfrak{p}(H):=\mathfrak{l}(H)+\mathfrak{u}_+(H)$
\end{center}
is a Levi decomposition of a parabolic subalgebra of $\mathfrak{g}$.

Let $\mathfrak{g}'$ be a reductive subalgebra of $\mathfrak{g}$, and $\mathfrak{p}$ a parabolic subalgebra of $\mathfrak{g}$.
\begin{definition}
We say $\mathfrak{p}$ is $\mathfrak{g}'$-$\emph{compatible}$ if there exists a hyperbolic element $H$ of $\mathfrak{g}'$ such that $\mathfrak{p}=\mathfrak{p}(H)$.
\end{definition}
If $\mathfrak{p}=\mathfrak{l}+\mathfrak{u}_+$ is $\mathfrak{g}'$-compatible, then $\mathfrak{p}':=\mathfrak{p}\cap\mathfrak{g}'$ becomes a parabolic subalgebra of $\mathfrak{g}'$ with Levi decomposition
\begin{center}
$\mathfrak{p}'=\mathfrak{l}'+\mathfrak{u}'_+:=(\mathfrak{l}\cap\mathfrak{g}')+(\mathfrak{u}_+\cap\mathfrak{g}')$.
\end{center}
\begin{proposition}(Proposition 3.8 [\textbf{Ko}])
If $\mathfrak{p}$ is $\mathfrak{g}'$-compatible, then $G'P$ is closed in $G$ and the restriction $\mathrm{Res}_{\mathfrak{g}'}^\mathfrak{g}X$ is discretely decomposable for any simple object $X$ in $\mathcal{O}^\mathfrak{p}$.
\end{proposition}

\begin{center}
\textbf{3.2 $\mathfrak{g}_2$-compatible parabolic subalgebra of $so(7,\mathbb{C})$}
\end{center}

We begin to find all the $\mathfrak{g}_2$-compatible standard parabolic subalgebras of $so(7,\mathbb{C})$.
\begin{lemma}
Let $\mathfrak{g}$ be a complex semi-simple Lie algebra with a reductive subalgebra $\mathfrak{g}'$. Suppose that $\mathfrak{p}$ is a parabolic subalgebra of $\mathfrak{g}$. Take an arbitrary maximal toral subalgebra $\mathfrak{h}'$ of $\mathfrak{p}':=\mathfrak{p}\cap\mathfrak{g}'$. Then $\mathfrak{p}$ is $\mathfrak{g}'$-compatible if and only if there exists a hyperbolic element $H$ in $\mathfrak{h}'$ such that $\mathfrak{p}=\mathfrak{p}(H)$.
\end{lemma}
\begin{proof}
The ``if'' part follows the definition immediately. Now if $\mathfrak{p}$ is a $\mathfrak{g}'$-compatible parabolic subalgebra, there is a hyperbolic element $H'\in\mathfrak{g}'$ such that $\mathfrak{p}=\mathfrak{p}(H')$. Moreover, $H'$ is a semi-simple element in $\mathfrak{p}$, so in $\mathfrak{p}'$, and hence there exists a maximal toral subalgebra $\mathfrak{t}'$ of $\mathfrak{p}'$ such that $H'\in\mathfrak{t}'$. Then $\mathfrak{h}'$ is conjugate to $\mathfrak{t}'$ by an element $x\in\mathrm{Int}(\mathfrak{p}')$, so $\mathrm{Ad}(x)H'\in\mathfrak{h}'$ which is also a hyperbolic element. Therefore, $\mathfrak{p}=\mathrm{Ad}(x)(\mathfrak{p})=\mathfrak{p}(\mathrm{Ad}(x)H')$. Let $H=\mathrm{Ad}(x)H'$, and then the ``only if'' part is proved.
\end{proof}
We know that $so(7,\mathbb{C})$ has eight standard parabolic subalgebras, which are corresponding to $\phi$, $\{\varepsilon_1-\varepsilon_2\}$, $\{\varepsilon_2-\varepsilon_3\}$, $\{\varepsilon_3\}$, $\{\varepsilon_1-\varepsilon_2, \varepsilon_2-\varepsilon_3\}$, $\{\varepsilon_1-\varepsilon_2, \varepsilon_3\}$, $\{\varepsilon_2-\varepsilon_3, \varepsilon_3\}$, and $\Delta(so(7,\mathbb{C}))$.
\begin{proposition}
There are four $\mathfrak{g}_2$-compatible standard parabolic subalgebras of $so(7,\mathbb{C})$: $\mathfrak{b}_{so(7,\mathbb{C})}$, $\mathfrak{p}_{\{\varepsilon_2-\varepsilon_3\}}$, $\mathfrak{p}_{\{\varepsilon_1-\varepsilon_2, \varepsilon_3\}}$ and $\mathfrak{p}_{\Delta(so(7,\mathbb{C}))}$.
\end{proposition}
\begin{proof}
By Lemma 3.6, we only need to take the hyperbolic elements in $\mathfrak{h}_{\mathfrak{g}_2}$. Let $H=aH_{\alpha_1}+bH_{\alpha_2}=aH_1+(b-a)H_2+(2a-b)H_3\in\mathfrak{h}_{\mathfrak{g}_2}$ for some $a,b\in\mathbb{C}$ be a hyperbolic element in $\mathfrak{g}_2$. Since $(\varepsilon_1-\varepsilon_2)(H)=\varepsilon_3(H)=2a-b$, if $\mathfrak{p}_\Pi$ for $\Pi\subseteq\Delta(so(7,\mathbb{C}))$ is $\mathfrak{g}_2$-compatible, $\varepsilon_1-\varepsilon_2\in\Pi$ if and only if $\varepsilon_3\in\Pi$. Therefore, neither of $\mathfrak{p}_{\{\varepsilon_1-\varepsilon_2\}}$, $\mathfrak{p}_{\{\varepsilon_3\}}$, $\mathfrak{p}_{\{\varepsilon_1-\varepsilon_2, \varepsilon_2-\varepsilon_3\}}$, $\mathfrak{p}_{\{\varepsilon_2-\varepsilon_3, \varepsilon_3\}}$ can be $\mathfrak{g}_2$-compatible. Take the hyperbolic element $H_\phi=4H_{\alpha_1}+7H_{\alpha_2}=4H_1+3H_2+H_3\in\mathfrak{h}_{\mathfrak{g}_2}$, and one will easily check that $\mathfrak{b}_{so(7,\mathbb{C})}=\mathfrak{p}(H_\phi)$. In fact, we have $\mathrm{ad}H_\phi(\mathfrak{h}_{\mathfrak{g}_2})\equiv0$, $(\varepsilon_1-\varepsilon_2)(H_\phi)=4-3=1>0$, $(\varepsilon_2-\varepsilon_3)(H_\phi)=3-1=2>0$ and $\varepsilon_3(H_\phi)=1>0$. Hence, by definition, $\mathfrak{b}_{so(7,\mathbb{C})}$ is $\mathfrak{g}_2$-compatible. Similarly, take the hyperbolic element $H'=2H_{\alpha_1}+3H_{\alpha_2}=2H_1+H_2+H_3\in\mathfrak{h}_{\mathfrak{g}_2}$, and one will easily check that $\mathfrak{p}_{\{\varepsilon_2-\varepsilon_3\}}=\mathfrak{p}(H')$ is $\mathfrak{g}_2$-compatible. And if taking the hyperbolic element $H''=H_{\alpha_1}+2H_{\alpha_2}=H_1+H_2\in\mathfrak{h}_{\mathfrak{g}_2}$, one will easily check that $\mathfrak{p}_{\{\varepsilon_1-\varepsilon_2, \varepsilon_3\}}=\mathfrak{p}(H'')$ is $\mathfrak{g}_2$-compatible. Finally, $\mathfrak{p}_{\Delta(so(7,\mathbb{C}))}=so(7,\mathbb{C})$ is obvious $\mathfrak{g}_2$-compatible because we just take $H_0=0$ to be the required hyperbolic element.
\end{proof}
The result of Proposition 3.7 was also showed by T.Milev and P.Somberg (Corollary 5.3 [\textbf{MS1}]). We only focus on this four $\mathfrak{g}_2$-compatible parabolic subalgebras from now on. In the next section, we shall begin to study decomposition of generalized Verma module in the generalized BGG category $\mathcal{O}^{\mathfrak{p}}$, which is defined to be $M_\mathfrak{p}^{so(7,\mathbb{C})}(\lambda):=U(so(7,\mathbb{C}))\bigotimes_{U(\mathfrak{p})}F_\lambda$ where $\mathrm{Res}_{\mathfrak{h}_{so(7,\mathbb{C})}\cap[\mathfrak{l},\mathfrak{l}]}^{\mathfrak{h}_{so(7,\mathbb{C})}}\lambda$ is a dominant integral weight and $F_\lambda$ is the finite dimensional simple $\mathfrak{l}$-module with highest weight $\lambda\in\Lambda^+(\mathfrak{l})$. Here, $F_\lambda$ is inflated to a $\mathfrak{p}$-module via the projection $\mathfrak{p}\to\mathfrak{p}/\mathfrak{u}_+\cong\mathfrak{l}$. One of extreme cases is $\mathfrak{p}=\mathfrak{b}_{so(7,\mathbb{C})}$. In this case, generalized Verma modules are just Verma modules. However, we shan't discuss the generalized Verma modules attached to $\mathfrak{p}_{\Delta(so(7,\mathbb{C}))}=so(7,\mathbb{C})$ because in this extreme case, $M_{\mathfrak{p}_{\Delta(so(7,\mathbb{C}))}}^{so(7,\mathbb{C})}(\lambda)$ is nothing but the finite dimensional simple $so(7,\mathbb{C})$-module with highest weight $\lambda\in\Lambda^+(so(7,\mathbb{C}))$, whose branching law to $\mathfrak{g}_2$ has already been solved (Theorem 3.4 [\textbf{M}]).

\section{Decomposition in the Grothendieck Group}

\begin{center}
\textbf{4.1 T.Kobayashi's Method}
\end{center}

In this section, we state an important theorem given by T.Kobayashi which offered us a main method to compute decompositions in the next three sections.

Let $\mathfrak{g}'$ be a reductive subalgebra of a complex semi-simple Lie algebra $\mathfrak{g}$, and $\mathfrak{p}$ a parabolic subalgebra of $\mathfrak{g}$. Let $\mathfrak{p}=\mathfrak{l}+\mathfrak{u}_+$ be a $\mathfrak{g}'$-compatible parabolic subalgebra of $\mathfrak{g}$ defined by a hyperbolic element $H\in\mathfrak{g}'$. We take a Cartan subalgebra $\mathfrak{h}_{\mathfrak{g}'}$ of $\mathfrak{g}'$ such that $H\in\mathfrak{h}_{\mathfrak{g}'}$, and extend it to a Cartan subalgebra $\mathfrak{h}_\mathfrak{g}$ of $\mathfrak{g}$. Clearly, $\mathfrak{h}_\mathfrak{g}\subseteq\mathfrak{l}$ and $\mathfrak{h}_{\mathfrak{g}'}\subseteq\mathfrak{l}'$ where $\mathfrak{l}'=\mathfrak{l}\cap\mathfrak{g}$ is the Levi factor of the parabolic subalgebra $\mathfrak{p}'=\mathfrak{p}\cap\mathfrak{g}'$ of $\mathfrak{g}'$.

We recall that $F_\lambda$ denotes the finite dimensional simple module of $\mathfrak{l}$ with highest weight $\lambda\in\Lambda^+(\mathfrak{l})$. Likewise, let $F'_\delta$ denote that of $\mathfrak{l}'$ for $\delta\in\Lambda^+(\mathfrak{l}')$.

Given a vector space $V$ we denote by $S(V)=\displaystyle{\bigoplus_{k=0}^{+\infty}}S^k(V)$ the symmetric tensor algebra over $V$. We extend the adjoint action of $\mathfrak{l}'$ on $\mathfrak{u}_-/\mathfrak{u}_-\cap\mathfrak{g}'$ to $S(\mathfrak{u}_-/\mathfrak{u}_-\cap\mathfrak{g}')$. We set
\begin{center}
$m(\delta;\lambda):=\dim_\mathbb{C}\hom_{\mathfrak{l}'}(F'_\delta,\mathrm{Res}_{\mathfrak{l}'}^\mathfrak{l}F_\lambda\bigotimes S(\mathfrak{u}_-/\mathfrak{u}_-\cap\mathfrak{g}'))$.
\end{center}
We denote $M_\mathfrak{p}^\mathfrak{g}$ to be the parabolic Verma module of $\mathfrak{g}$ attached to its parabolic subalgebra $\mathfrak{p}$. Likewise, let $M_{\mathfrak{p}'}^{\mathfrak{g}'}$ denote that of $\mathfrak{g}'$ attached to its parabolic subalgebra $\mathfrak{p}'$.
\begin{theorem}(Theorem 3.10 [\textbf{Ko}])
Suppose that $\mathfrak{p}=\mathfrak{l}+\mathfrak{u}_+$ is a $\mathfrak{g}'$-compatible parabolic subalgebra of $\mathfrak{g}$, and $\lambda\in\Lambda^+(\mathfrak{l})$.
\begin{enumerate}[(1)]
\item $m(\delta;\lambda)<+\infty$ for all $\delta\in\Lambda^+(\mathfrak{l}')$.
\item In the Grothendieck group of $\mathcal{O}^{\mathfrak{p}'}$, we have the following isomorphism:
\begin{center}
$\mathrm{Res}_{\mathfrak{g}'}^\mathfrak{g}M_\mathfrak{p}^\mathfrak{g}(\lambda)\cong\displaystyle{\bigoplus_{\delta\in\Lambda^+(\mathfrak{l}')}}m(\delta;\lambda)
M_{\mathfrak{p}'}^{\mathfrak{g}'}(\delta)$.
\end{center}
\end{enumerate}
\end{theorem}
In Section 4.2, 4.3 and 4.4, we shall make use of the method introduced in Section 4.1 to obtain decompositions as in Theorem 4.1. Our main aim is nothing but to compute $m(\delta;\lambda)$. Hence, it includes three steps.
\begin{enumerate}[Step 1:]
\item to compute $\mathrm{Res}_{\mathfrak{l}'}^\mathfrak{l}F_\lambda$;
\item to compute $S(\mathfrak{u}_-/\mathfrak{u}_-\cap\mathfrak{g}_2)$;
\item to compute $\dim_\mathbb{C}\hom_{\mathfrak{l}'}(F'_\delta,\mathrm{Res}_{\mathfrak{l}'}^\mathfrak{l}F_\lambda\otimes S(\mathfrak{u}_-/\mathfrak{u}_-\cap\mathfrak{g}_2))$.
\end{enumerate}
In Section 4.2 where the generalized Verma module is just the standard Verma module, the computation is not complicate because $F'_\delta$ will be only one dimensional. However, things are become much more complicate in the latter two cases. One will see that $F'_\delta$ is isomorphic to some finite dimensional module of $sl(2,\mathbb{C})$.

\begin{center}
\textbf{4.2 $M_{\mathfrak{b}_{so(7,\mathbb{C})}}^{so(7,\mathbb{C})}(\lambda)$}
\end{center}

Let $\mathfrak{p}=\mathfrak{b}_{so(7,\mathbb{C})}$, then $\mathfrak{p}'=\mathfrak{p}\cap\mathfrak{g}_2=\mathfrak{b}_{\mathfrak{g}_2}$, $\mathfrak{l}=\mathfrak{h}_{so(7,\mathbb{C})}$, and $\mathfrak{l}'=\mathfrak{h}_{\mathfrak{g}_2}=\mathrm{Span}_\mathbb{C}\{H_1-H_2+2H_3,H_2-H_3\}$. Given $\lambda=a\varepsilon_1+b\varepsilon_2+c\varepsilon_3$ for some $a,b,c\in\mathbb{C}$, $\mathrm{Res}_{\mathfrak{h}_{\mathfrak{g}_2}}^{\mathfrak{h}_{so(7,\mathbb{C})}}\lambda=(2a+b+c)\alpha_1+(a+b)\alpha_2$. Thus, $F_\lambda$ is an 1-dimension simple module of $\mathfrak{h}_{so(7,\mathbb{C})}$ and
\begin{center}
$\mathrm{Res}_{\mathfrak{h}_{\mathfrak{g}_2}}^{\mathfrak{h}_{so(7,\mathbb{C})}}F_\lambda=F'_{(2a+b+c)\alpha_1+(a+b)\alpha_2}$.
\end{center}
On the other hand, it is not hard to compute that $\mathfrak{u}_-/\mathfrak{u}_-\cap\mathfrak{g}_2=\mathrm{Span}_\mathbb{C}\{v_{-\alpha_1},\\v_{-\alpha_1-\alpha_2},v_{-2\alpha_1-\alpha_2}\}\cong F'_{-\alpha_1}\oplus F'_{-\alpha_1-\alpha_2}\oplus F'_{-2\alpha_1-\alpha_2}$ as $\mathfrak{h}_{\mathfrak{g}_2}$-module. Therefore,
\begin{center}
$S(\mathfrak{u}_-/\mathfrak{u}_-\cap\mathfrak{g}_2)=\displaystyle{\bigoplus_{i,j,k\in\mathbb{N}}}F'_{-(2i+j+k)\alpha_1-(i+j)\alpha_2}$.
\end{center}
Combine the two equations above, and we have
\begin{center}
$\mathrm{Res}_{\mathfrak{h}_{\mathfrak{g}_2}}^{\mathfrak{h}_{so(7,\mathbb{C})}}F_\lambda\otimes S(\mathfrak{u}_-/\mathfrak{u}_-\cap\mathfrak{g}_2)\cong\displaystyle{\bigoplus_{i,j,k\in\mathbb{N}}}F'_{(2a+b+c-2i-j-k)\alpha_1+(a+b-i-j)\alpha_2}$.
\end{center}
Now let $\delta=u\alpha_1+v\alpha_2$ for some $u,v\in\mathbb{C}$ such that $F'_{\delta}$ is an 1-dimension simple module of $\mathfrak{h}_{\mathfrak{g}_2}$. Then according to the definition of $m(\delta,\lambda)$, we have
\begin{center}
$m(\delta,\lambda)=\mathrm{Card}\{(i,j,k)\in\mathbb{N}^3\mid 2i+j+k=2a+b+c-u,i+j=a+b-v\}$.
\end{center}
Based on the discussion above, we quickly obtain the formula for $m(\delta,\lambda)$.
\begin{lemma}
With the notations above, and let $\mu=\mathrm{Res}_{\mathfrak{g}_2}^{so(7,\mathbb{C})}\lambda-\delta$, then $m(\delta,\lambda)=1+\min\{\mu(H_{3\alpha_1+\alpha_2}),\mu(H_{3\alpha_1+2\alpha_2})\}$ if and only if $\mu(H_{3\alpha_1+\alpha_2})\in\mathbb{N}$ and $\mu(H_{3\alpha_1+2\alpha_2})\in\mathbb{N}$; otherwise, $m(\delta,\lambda)=0$.
\end{lemma}
\begin{proof}
We have already known that $m(\delta,\lambda)\neq0$ if and only if there exist $i,j,k\in\mathbb{N}$ such that $2i+j+k=2a+b+c-u$ and $i+j=a+b-v$; equivalently, $i+k=a+c-u+v$ and $i+j=a+b-v$. Because $j,k\geq0$, $0\leq i\leq\min\{a+c-u+v,a+b-v\}$. Moreover, once $0\leq i\leq\min\{a+c-u+v,a+b-v\}, i\in\mathbb{N}$ is fixed, there exists a unique $(j,k)\in\mathbb{N}^2$ such that $i+k=a+c-u+v$ and $i+j=a+b-v$ hold. Thus, $m(\delta,\lambda)=\mathrm{Card}\{(i,j,k)\in\mathbb{N}^3\mid i+k=a+c-u+v,i+j=a+b-v\}=1+min\{a+c-u+v,a+b-v\}$. But $\mu(H_{3\alpha_1+\alpha_2})=a+c-u+v$ and $\mu(H_{3\alpha_1+2\alpha_2})=a+b-v$, from which the conclusion follows.
\end{proof}
Apply Theorem 4.1, the decomposition of $M_{\mathfrak{b}_{so(7,\mathbb{C})}}^{so(7,\mathbb{C})}(\lambda)$ in the Grothendieck Group of $\mathcal{O}^{\mathfrak{b}_{\mathfrak{g}_2}}$ is obtained immediately.
\begin{proposition}
Let $\lambda\in\mathfrak{h}_{so(7,\mathbb{C})}^*$ and $\delta\in\mathfrak{h}_{\mathfrak{g}_2}^*$. Denote $\mu=\mathrm{Res}_{\mathfrak{g}_2}^{so(7,\mathbb{C})}\lambda-\delta$. Then
\begin{center}
$\mathrm{Res}_{\mathfrak{g}_2}^{so(7,\mathbb{C})}M_{\mathfrak{b}_{so(7,\mathbb{C})}}^{so(7,\mathbb{C})}(\lambda)
=\displaystyle{\bigoplus_{\mbox{\tiny$\begin{array}{c}\mu(H_{3\alpha_1+\alpha_2})\in\mathbb{N}\\ \mu(H_{3\alpha_1+2\alpha_2})\in\mathbb{N}\end{array}$}}}
(1+\min\{\mu(H_{3\alpha_1+\alpha_2}),\mu(H_{3\alpha_1+2\alpha_2})\})M_{\mathfrak{b}_{\mathfrak{g}_2}}^{\mathfrak{g}_2}(\delta)$
\end{center}
in the Grothendieck Group of $\mathcal{O}^{\mathfrak{b}_{\mathfrak{g}_2}}$.
\end{proposition}
\begin{proof}
By Theorem 4.1 and Lemma 4.2, the conclusion is proved.
\end{proof}

\begin{center}
\textbf{4.3 $M_{\mathfrak{p}_{\{\varepsilon_2-\varepsilon_3\}}}^{so(7,\mathbb{C})}(\lambda)$}
\end{center}

We consider $\mathfrak{p}=\mathfrak{p}_{\{\varepsilon_2-\varepsilon_3\}}$. In this case, $\mathfrak{p}'=\mathfrak{p}_{\{\alpha_2\}}$, $\mathfrak{l}=\mathfrak{h}_{so(7,\mathbb{C})}\oplus\mathbb{C}X_{\varepsilon_2-\varepsilon_3}\oplus\mathbb{C}X_{\varepsilon_3-\varepsilon_2}$, and $\mathfrak{l}'=\mathfrak{h}_{\mathfrak{g}_2}\oplus\mathbb{C}X_{\alpha_2}\oplus\mathbb{C}X_{-\alpha_2}$. We write $\mathfrak{l}'=\mathrm{Span}_\mathbb{C}\{H_{2\alpha_1+\alpha_2},H_{\alpha_2},X_{\alpha_2},\\X_{-\alpha_2}\}=\mathrm{Span}_\mathbb{C}
\{2H_1+H_2+H_3,H_2-H_3,X_{\alpha_2},X_{-\alpha_2}\}$, and it is obvious that $\mathfrak{l}'\cong gl(2,\mathbb{C})$ given by $H_{2\alpha_1+\alpha_2}\to\left(\begin{array}{cc}1&0\\0&1\end{array}\right)$, $H_{\alpha_2}\to\left(\begin{array}{cc}1&0\\0&-1\end{array}\right)$, $X_{\alpha_2}\to\left(\begin{array}{cc}0&1\\0&0\end{array}\right)$, and $X_{-\alpha_2}\to\left(\begin{array}{cc}0&0\\1&0\end{array}\right)$. In particular, $\mathrm{Span}_\mathbb{C}\{H_{\alpha_2},X_{\alpha_2},X_{-\alpha_2}\}\cong sl(2,\mathbb{C})$ and $\mathbb{C}H_{2\alpha_1+\alpha_2}$ is the center of $\mathfrak{l}'$.

From now on, we denote $F(n)$ to be the finite dimensional simple module of $gl(2,\mathbb{C})$ with highest weight $n$, whose center act as 0. Moreover, we denote $\rho_0$ to be the 1-dimensional simple module of $gl(2,\mathbb{C})$ where $sl(2,\mathbb{C})$ act as 0 and $\left(\begin{array}{cc}1&0\\0&1\end{array}\right)$ acts as multiplication by 2.

Now let $\lambda=a\varepsilon_1+b\varepsilon_2+c\varepsilon_3$ for some $a,b,c\in\mathbb{C}$. Here, we require $\lambda$ to be $[\mathfrak{l},\mathfrak{l}]$-dominant integral, so $\lambda(H_{\varepsilon_2-\varepsilon_3})=b-c\in\mathbb{N}$. We compute that $\lambda(H_{2\alpha_1+\alpha_2})=2a+b+c$ and $\lambda(H_{\alpha_2})=b-c$. Hence,
\begin{center}
$\mathrm{Res}_{\mathfrak{l}'}^\mathfrak{l}F_\lambda\cong F(b-c)\otimes\frac{2a+b+c}{2}\rho_0$.
\end{center}
On the other hand, $\mathfrak{u}_-/\mathfrak{u}_-\cap\mathfrak{g}_2=\mathrm{Span}_\mathbb{C}\{v_{-\alpha_1},v_{-\alpha_1-\alpha_2},v_{-2\alpha_1-\alpha_2}\}$. Check the action of $H_{2\alpha_1+\alpha_2}$, $H_{\alpha_2}$, $X_{\alpha_2}$ and $X_{-\alpha_2}$ on the weight vectors $v_{-\alpha_1}$, $v_{-\alpha_1-\alpha_2}$ and $v_{-2\alpha_1-\alpha_2}$, we have $\mathrm{Span}_\mathbb{C}\{v_{-\alpha_1},v_{-\alpha_1-\alpha_2}\}\cong F(1)\otimes-\frac{1}{2}\rho_0$ and $\mathrm{Span}_\mathbb{C}\{v_{-2\alpha_1-\alpha_2}\}\cong-\rho_0$. Thus,
\begin{center}
$S(\mathfrak{u}_-/\mathfrak{u}_-\cap\mathfrak{g}_2)\cong\displaystyle{\bigoplus_{i,j\in\mathbb{N}}}S^i(F(1)\otimes-\frac{1}{2}\rho_0)\otimes S^j(-\rho_0)\cong \displaystyle{\bigoplus_{i,j\in\mathbb{N}}}F(i)\otimes-(\frac{i}{2}+j)\rho_0$.
\end{center}
Here, we use the fact that $S^i(F(1))\cong F(i)$. Combine the two equations above, and we have
\begin{center}
$\mathrm{Res}_{\mathfrak{l}'}^\mathfrak{l}F_\lambda\otimes S(\mathfrak{u}_-/\mathfrak{u}_-\cap\mathfrak{g}_2)\cong\displaystyle{\bigoplus_{i,j\in\mathbb{N}}}F(i)\otimes F(b-c)\otimes(\frac{2a+b+c}{2}-\frac{i}{2}-j)\rho_0$.
\end{center}
But $F(i)\otimes F(b-c)$ is not necessarily a simple $gl(2,\mathbb{C})$-module. Luckily, by Littlewood-Richardson theorem (Theorem 9.74 [\textbf{K}]), we have $F(i)\otimes F(b-c)\cong\displaystyle{\bigoplus_{\mbox{\tiny$\begin{array}{c}|b-c-i|\leq k\leq b-c+i\\b-c+i-k\equiv0(\mathrm{mod} 2)\end{array}$}}}F(k)$. Then we obtain a direct sum with each summand a simple $\mathfrak{l}'$-module.
\begin{center}
$\mathrm{Res}_{\mathfrak{l}'}^\mathfrak{l}F_\lambda\otimes S(\mathfrak{u}_-/\mathfrak{u}_-\cap\mathfrak{g}_2)\cong\displaystyle{\bigoplus_{i,j\in\mathbb{N}}}\displaystyle{\bigoplus_{\mbox{\tiny$\begin{array}{c}|b-c-i|
\leq k\leq b-c+i\\b-c+i-k\equiv0(\mathrm{mod} 2)\end{array}$}}}F(k)\otimes(\frac{2a+b+c}{2}-\frac{i}{2}-j)\rho_0$.
\end{center}
Now let $\delta=u\alpha_1+v\alpha_2$ for some $u,v\in\mathbb{C}$. Here, we require $\delta$ to be $[\mathfrak{l}',\mathfrak{l}']$-dominant integral, so $\delta(H_{\alpha_2})=2v-u\in\mathbb{N}$. And since $\delta(H_{2\alpha_1+\alpha_2})=u$, $F'_\delta\cong F(2v-u)\otimes\frac{u}{2}\rho_0$.
\begin{lemma}
With the notations above, and let $\sigma=\mathrm{Res}_{\mathfrak{g}_2}^{so(7,\mathbb{C})}\lambda+\delta$ and $\mu=\mathrm{Res}_{\mathfrak{g}_2}^{so(7,\mathbb{C})}\lambda-\delta$. Then $m(\delta,\lambda)\neq0$ if and only if $\mu(H_{2\alpha_1+\alpha_2})\in\mathbb{N}$, $\mu(H_{3\alpha_1+\alpha_2})\in\mathbb{Z}$, and $|\mu(H_{\alpha_2})|\leq\mu(H_{2\alpha_1+\alpha_2})$. In this case, \[m(\delta,\lambda)=1+\frac{\min\{\mu(H_{2\alpha_1+\alpha_2}),\sigma(H_{\alpha_2})\}-|\mu(H_{\alpha_2})|}{2};\] otherwise, $m(\delta,\lambda)=0$.
\end{lemma}
\begin{proof}
According to the discussion above, $m(\delta,\lambda)\neq0$ if and only if there exist $i,j,k\in\mathbb{N}$ such that $|b-c-i|\leq k\leq b-c+i$, $b-c+i-k\equiv0(\mathrm{mod} 2)$, $k=2v-u$ and $\frac{2a+b+c}{2}-\frac{i}{2}-j=\frac{u}{2}$ if and only if there exist $i,j\in\mathbb{N}$ such that $|b-c+u-2v|\leq i\leq b-c-u+2v$, $b-c+u-2v+i\equiv0(\mathrm{mod} 2)$ and $i+2j=2a+b+c-u$.

Now suppose that $|b-c+u-2v|\leq i\leq b-c-u+2v$, $b-c+u-2v+i\equiv0(\mathrm{mod} 2)$ and $i+2j=2a+b+c-u$ hold for some $i,j\in\mathbb{N}$, then $2a+b+c-u=i+2j\in\mathbb{N}$ and $|b-c+u-2v|\leq i\leq i+2j\leq 2a+b+c-u$. Moreover, $b-c+u-2v+i\equiv0\equiv2j\equiv2a+b+c-u-i(\mathrm{mod} 2)$; hence, $2a+2c-2u+2v\equiv2i\equiv0(\mathrm{mod} 2)$ which is equivalent to $a+c-u+v\in\mathbb{Z}$.

Conversely, suppose that $2a+b+c-u\in\mathbb{N}$, $a+c-u+v\in\mathbb{Z}$ and $|b-c+u-2v|\leq2a+b+c-u$ hold, then we just take $i=|b-c+u-2v|$ and $j=\frac{2a+b+c-u-|b-c+u-2v|}{2}$. Here, $i\leq b-c-u+2v$ because $b-c,2v-u\in\mathbb{N}$. Also, $j$ must be an integer because $a+c-u+v\in\mathbb{Z}$ guarantees that $2a+b+c-u$ and $b-c+u-2v$ have same parity and so do $|b-c+u-2v|$ and $b-c+u-2v$. Thus, $2a+b+c-u-|b-c+u-2v|$ is always an even integer. One can check immediately that $i,j\in\mathbb{N}$, $|b-c+u-2v|\leq i\leq b-c-u+2v$, $b-c+u-2v+i\equiv0(\textrm{mod}2)$ and $i+2j=2a+b+c-u$. This shows that $m(\delta,\lambda)\neq0$ if and only if $2a+b+c-u\in\mathbb{N}$, $a+c-u+v\in\mathbb{Z}$ and $|b-c+u-2v|\leq2a+b+c-u$, which is just $\mu(H_{2\alpha_1+\alpha_2})\in\mathbb{N}$, $\mu(H_{3\alpha_1+\alpha_2})\in\mathbb{Z}$ and $|\mu(H_{\alpha_2})|\leq\mu(H_{2\alpha_1+\alpha_2})$.

We know that $m(\delta,\lambda)=\mathrm{Card}\{(i,j)\in\mathbb{N}^2\mid|b-c+u-2v|\leq i\leq b-c-u+2v,b-c+u-2v+i\equiv0(\mathrm{mod}2),i+2j=2a+b+c-u\}$. By these three conditions, $i$ may only be $|b-c+u-2v|$, $|b-c+u-2v|+2$, $|b-c+u-2v|+4$, $\cdots$, $\min\{2a+b+c-u,b-c-u+2v\}$. Moreover, if $i$ is fixed, there exists a unique $j\in\mathbb{N}$ such that $i+2j=2a+b+c-u$. Hence, $m(\delta,\lambda)$ is just equal to the number of the choices for $i$, i.e. $1+\frac{\min\{2a+b+c-u, b-c-u+2v\}-|b-c+u-2v|}{2}$ which is just $1+\frac{\min\{\mu(H_{2\alpha_1+\alpha_2}),\sigma(H_{\alpha_2})\}-|\mu(H_{\alpha_2})|}{2}$.
\end{proof}
It is not difficult to check that some conditions can be deduced from others, so after cancelling some of them, we only need three conditions: $\delta(H_{\alpha_2})\in\mathbb{N}$, $\mu(H_{3\alpha_1+\alpha_2})\in\mathbb{N}$ and $\mu(H_{3\alpha_1+2\alpha_2})\in\mathbb{N}$.

Apply Theorem 4.1, the decomposition of $M_{\mathfrak{p}_{\{\varepsilon_2-\varepsilon_3\}}}^{so(7,\mathbb{C})}(\lambda)$ in the Grothendieck Group of $\mathcal{O}^{\mathfrak{p}_{\{\alpha_2\}}}$ is obtained immediately.
\begin{proposition}
Let $\lambda\in\mathfrak{h}_{so(7,\mathbb{C})}^*$ satisfying $\lambda(H_{\varepsilon_2-\varepsilon_3})\in\mathbb{N}$ and $\delta\in\mathfrak{h}_{\mathfrak{g}_2}^*$. Denote $\sigma=\mathrm{Res}_{\mathfrak{g}_2}^{so(7,\mathbb{C})}\lambda+\delta$ and $\mu=\mathrm{Res}_{\mathfrak{g}_2}^{so(7,\mathbb{C})}\lambda-\delta$. Then
\begin{center}
$\mathrm{Res}_{\mathfrak{g}_2}^{so(7,\mathbb{C})}M_{\mathfrak{p}_{\{\varepsilon_2-\varepsilon_3\}}}^{so(7,\mathbb{C})}(\lambda)= \displaystyle{\bigoplus_{\mbox{\tiny$\begin{array}{c}\delta(H_{\alpha_2})\in\mathbb{N}\\ \mu(H_{3\alpha_1+\alpha_2})\in\mathbb{N}\\ \mu(H_{3\alpha_1+2\alpha_2})\in\mathbb{N}\end{array}$}}}
(1+\frac{A-|\mu(H_{\alpha_2})|}{2})M_{\mathfrak{p}_{\{\alpha_2\}}}^{\mathfrak{g}_2}(\delta)$
\end{center}
in the Grothendieck Group of $\mathcal{O}^{\mathfrak{p}_{\{\alpha_2\}}}$, where $A=\min\{\mu(H_{2\alpha_1+\alpha_2}),\sigma(H_{\alpha_2})\}$.
\end{proposition}
\begin{proof}
By Theorem 4.1 and Lemma 4.4, the conclusion is proved.
\end{proof}

\begin{center}
\textbf{4.4 $M_{\mathfrak{p}_{\{\varepsilon_1-\varepsilon_2,\varepsilon_3\}}}^{so(7,\mathbb{C})}(\lambda)$}
\end{center}

We turn to the last case. In this case $\mathfrak{p}=\mathfrak{p}_{\{\varepsilon_1-\varepsilon_2,\varepsilon_3\}}$, $\mathfrak{p}'=\mathfrak{p}_{\{\alpha_1\}}$, $\mathfrak{l}=\mathfrak{h}_{so(7,\mathbb{C})}\oplus\mathbb{C}X_{\varepsilon_1-\varepsilon_2}\oplus\mathbb{C}X_{\varepsilon_2-\varepsilon_1}\oplus\mathbb{C}
X_{\varepsilon_3}\oplus\mathbb{C}X_{\varepsilon_{-3}}=\mathbb{C}H_{\varepsilon_1+\varepsilon_2}\oplus \mathrm{Span}_\mathbb{C}\{H_{\varepsilon_1-\varepsilon_2},X_{\varepsilon_1-\varepsilon_2},X_{\varepsilon_2-\varepsilon_1}\}\oplus \mathrm{Span}_\mathbb{C}\{H_{\varepsilon_3},X_{\varepsilon_3},X_{-\varepsilon_3}\}\cong\mathbb{C}\oplus sl(2,\mathbb{C})\oplus sl(2,\mathbb{C})$ as direct sum of ideals, and $\mathfrak{l}'=\mathfrak{h}_{\mathfrak{g}_2}\oplus\mathbb{C}X_{\alpha_1}\oplus\mathbb{C}X_{-\alpha_1}=\mathbb{C}H_{3\alpha_1+2\alpha_2}\oplus \mathrm{Span}_\mathbb{C}\{H_{\alpha_1},X_{\alpha_1},X_{\alpha_{-1}}\}\cong gl(2,\mathbb{C})$ given by $-H_{3\alpha_1+2\alpha_2}\to\left(\begin{array}{cc}1&0\\0&1\end{array}\right)$, $H_{\alpha_1}\to\left(\begin{array}{cc}1&0\\0&-1\end{array}\right)$, $X_{\alpha_1}\to\left(\begin{array}{cc}0&1\\0&0\end{array}\right)$, and $X_{-\alpha_1}\to\left(\begin{array}{cc}0&0\\1&0\end{array}\right)$. In particular, $\mathrm{Span}_\mathbb{C}\{H_{\alpha_1},X_{\alpha_1},X_{-\alpha_1}\}\cong sl(2,\mathbb{C})$ and $\mathbb{C}H_{3\alpha_1+2\alpha_2}$ is the center of $\mathfrak{l}'$. According to the construction in Section 2, $\mathfrak{l}'\cong\mathbb{C}\oplus sl(2,\mathbb{C})$ is embedded into $\mathfrak{l}\cong\mathbb{C}\oplus sl(2,\mathbb{C})\oplus sl(2,\mathbb{C})$ as follows: $\mathbb{C}\hookrightarrow\mathbb{C}$ and $sl(2,\mathbb{C})$ is embedded into $sl(2,\mathbb{C})\oplus sl(2,\mathbb{C})$ diagonally.

To avoid the similar calculation as in Section 4.3, we omit details of computation.

Let $\lambda=a\varepsilon_1+b\varepsilon_2+c\varepsilon_3$ for some $a,b,c\in\mathbb{C}$. Here, we require $\lambda$ to be $[\mathfrak{l},\mathfrak{l}]$-dominant integral, so $\lambda(H_{\varepsilon_1-\varepsilon_2})=\lambda(H_1-H_2)=a-b\in\mathbb{N}$ and $\lambda(H_{\varepsilon_3})=\lambda(2H_3)=2c\in\mathbb{N}$. Retain the notations $F(n)$ and $\rho_0$ in Section 4.3 for $\mathfrak{l}'$ now. We have $\mathfrak{u}_-/\mathfrak{u}_-\cap\mathfrak{g}_2=\mathrm{Span}_\mathbb{C}\{v_{-\alpha_1-\alpha_2},v_{-2\alpha_1-\alpha_2}\}$, and $\mathrm{Res}_{\mathfrak{l}'}^\mathfrak{l}F_\lambda\otimes S(\mathfrak{u}_-/\mathfrak{u}_-\cap\mathfrak{g}_2)\cong\displaystyle{\bigoplus_{k\in\mathbb{N}}\bigoplus_{\mbox{\tiny$\begin{array}{c}|a-b-2c|\leq i\leq a-b+2c\\a-b+2c-i\equiv0(\mathrm{mod} 2)\end{array}$}}\bigoplus_{\mbox{\tiny$\begin{array}{c}|i-k|\leq j\leq i+k\\i+k-j\equiv0(\mathrm{mod} 2)\end{array}$}}}F(j)\otimes(\frac{k-a-b}{2})\rho_0$.

Now let $\delta=u\alpha_1+v\alpha_2$ for some $u,v\in\mathbb{C}$. Here, we require $\delta$ to be $[\mathfrak{l}',\mathfrak{l}']$-dominant integral, so $\delta(H_{\alpha_1})=2u-3v\in\mathbb{N}$. And since $\delta(-H_{3\alpha_1+2\alpha_2})=-v$, $F'_\delta\cong F(2u-3v)\otimes(-\frac{v}{2})\rho_0$.
\begin{lemma}
With the notations above, and let $\mu=\mathrm{Res}_{\mathfrak{g}_2}^{so(7,\mathbb{C})}\lambda-\delta$. Then $m(\delta,\lambda)\neq0$ if and only if $\mu(H_{3\alpha_1+2\alpha_2})\in\mathbb{N}$, $\mu(H_{3\alpha_1+\alpha_2})\in\mathbb{Z}$, and $\max\{|\mu(\\H_{3\alpha_1+2\alpha_2})-\delta(H_{\alpha_1})|,|\lambda(H_{\varepsilon_1-\varepsilon_2}-H_{\varepsilon_3})|\}\leq
\min\{\mu(H_{3\alpha_1+2\alpha_2})+\delta(H_{\alpha_1}),\lambda(H_{\varepsilon_1-\varepsilon_2}+H_{\varepsilon_3})\}$. In this case, \[m(\delta,\lambda)=1+\frac{B-C}{2},\] where \[B=\min\{\mu(H_{3\alpha_1+2\alpha_2})+\delta(H_{\alpha_1}),\lambda(H_{\varepsilon_1-\varepsilon_2}+H_{\varepsilon_3})\},\] \[C=\max\{|\mu(H_{3\alpha_1+2\alpha_2})\\-\delta(H_{\alpha_1})|,|\lambda(H_{\varepsilon_1-\varepsilon_2}-H_{\varepsilon_3})|\};\]
otherwise, $m(\delta,\lambda)=0$.
\end{lemma}
\begin{proof}
The proof is similar to that of Lemma 4.4.
\end{proof}

Apply Theorem 4.1, the decomposition of $M_{\mathfrak{p}_{\{\varepsilon_1-\varepsilon_2,\varepsilon_3\}}}^{so(7,\mathbb{C})}(\lambda)$ in the Grothendie-\\ck Group of $\mathcal{O}^{\mathfrak{p}_{\{\alpha_1\}}}$ is obtained immediately.
\begin{proposition}
Let $\lambda\in\mathfrak{h}_{so(7,\mathbb{C})}^*$ satisfying $\lambda(H_{\varepsilon_1-\varepsilon_2}),\lambda(H_{\varepsilon_3})\in\mathbb{N}$ and $\delta\in\mathfrak{h}_{\mathfrak{g}_2}^*$. Denote $\mu=\mathrm{Res}_{\mathfrak{g}_2}^{so(7,\mathbb{C})}\lambda-\delta$. Then
\begin{center}
$\mathrm{Res}_{\mathfrak{g}_2}^{so(7,\mathbb{C})}M_{\mathfrak{p}_{\{\varepsilon_1-\varepsilon_2,\varepsilon_3\}}}^{so(7,\mathbb{C})}(\lambda)
=\displaystyle{\bigoplus_{\mbox{\tiny$\begin{array}{c}\mu(H_{3\alpha_1+2\alpha_2})\in\mathbb{N}\\ \mu(H_{3\alpha_1+\alpha_2})\in\mathbb{Z}\\C\leq B\end{array}$}}}(1+\frac{B-C}{2})M_{\mathfrak{p}_{\{\alpha_1\}}}^{\mathfrak{g}_2}(\delta)$
\end{center}
in the Grothendieck Group of $\mathcal{O}^{\mathfrak{p}_{\{\alpha_1\}}}$, where
\begin{center}
$B=\min\{\mu(H_{3\alpha_1+2\alpha_2})+\delta(H_{\alpha_1}),\lambda(H_{\varepsilon_1-\varepsilon_2}+H_{\varepsilon_3})\}$,\\
$C=\max\{|\mu(H_{3\alpha_1+2\alpha_2})-\delta(H_{\alpha_1})|,|\lambda(H_{\varepsilon_1-\varepsilon_2}-H_{\varepsilon_3})|\}$.
\end{center}
\end{proposition}
\begin{remark}
There is no need for us to add the condition $\delta(H_{\alpha_1})=2u-3v\in\mathbb{N}$ under $\bigoplus$ because the three conditions already imply it. In fact, since $a+b-v\in\mathbb{N}$, $a+c-u+v\in\mathbb{Z}$ and $a-b,2c\in\mathbb{N}$, $2u-3v=(a+b-v)-2(a+c-u+v)+(a-b)+2c\in\mathbb{Z}$. And $|a+b-2u+2v|\leq a+b+2u-4v$ implies $2u-3v\geq0$.
\end{remark}
\begin{proof}
By Theorem 4.1 and Lemma 4.6, the conclusion is proved.
\end{proof}

\section{Decomposition as $\mathfrak{g}_2$-module}

\begin{center}
\textbf{5.1 Some basic results in generalized BGG category $\mathcal{O}^\mathfrak{p}$}
\end{center}

In this section, we list some results in generalized BGG category $\mathcal{O}^\mathfrak{p}$, which will be useful in the next three sections.

Let $\mathfrak{g}$ be a complex semi-simple Lie algebra, and let $\mathfrak{h}_\mathfrak{g}$ be a Cartan subalgebra with dual space $\mathfrak{h}_\mathfrak{g}^*$.
\begin{proposition}
Let $\mathfrak{p}=\mathfrak{p}_\Pi$ for some $\Pi\subseteq\Delta(\mathfrak{g})$. Assume that $\lambda\in\mathfrak{h}_\mathfrak{g}^*$ is $[\mathfrak{l},\mathfrak{l}]$-dominant integral. If $\frac{2(\lambda+\rho(\mathfrak{g}),\beta)}{(\beta,\beta)}\notin\mathbb{Z}^+$ for all $\beta\in\Phi^+(\mathfrak{g})-\mathbb{Z}\Pi$, then $M_\mathfrak{p}^\mathfrak{g}(\lambda)$ is simple.
\end{proposition}
\begin{proof}
See Theorem 9.12 [\textbf{Hu2}].
\end{proof}
\begin{definition}
We say that $\mu$ is linked to $\nu$ if $\mu-\nu\in\mathbb{Z}\Phi(\mathfrak{g})$ and $\mu=\omega(\nu+\rho)-\rho$ for some $\omega\in W_\mathfrak{g}$; in other words, $\mu$ and $\nu$ lie in the same linkage class.
\end{definition}
\begin{remark}
Linkage class is an equivalence class. Moreover, the weights in a same linkage class correspond to a same infinitesimal character.
\end{remark}
\begin{proposition}
There exists a unique anti-dominant weight in each linkage class. Moreover, if $\lambda_1$ and $\lambda_2$ lie in the different linkage classes, then $M_\mathfrak{p}^\mathfrak{g}(\lambda_1)$ and $M_\mathfrak{p}^\mathfrak{g}(\lambda_2)$ have no non-split extension in $\mathcal{O}^\mathfrak{p}$.
\end{proposition}
\begin{proof}
The conclusion follows Theorem 4.9 [\textbf{Hu2}].
\end{proof}
\begin{proposition}
Let $L(\lambda)$ be the simple highest weight module with the highest weight $\lambda\in\mathfrak{h}_\mathfrak{g}^*$. Then $L(\lambda)$ has no non-split extension with itself in $\mathcal{O}^\mathfrak{p}$.
\end{proposition}
\begin{proof}
See Proposition 3.1(d) [\textbf{Hu2}].
\end{proof}
We have made full preparation for our calculation. For each case, we have two steps.
\begin{enumerate}[Step 1:]
\item to add conditions on the parameter $\lambda$ such that $M_\mathfrak{p}^{so(7,\mathbb{C})}(\lambda)$ is simple, the aim of which is that any sub-quotient of $M_\mathfrak{p}^{so(7,\mathbb{C})}(\lambda)$ as $\mathfrak{g}_2$-module lies in $\mathcal{O}^{\mathfrak{p}'}$.
\item to add conditions on the parameter $\lambda$ such that the direct summands of each decomposition are $\mathfrak{g}_2$-simple and have no non-split extensions with each other in $\mathcal{O}^{\mathfrak{p}'}$.
\end{enumerate}

\begin{center}
\textbf{5.2 $M_{\mathfrak{b}_{so(7,\mathbb{C})}}^{so(7,\mathbb{C})}(\lambda)$}
\end{center}

The Verma module $M_{\mathfrak{b}_{so(7,\mathbb{C})}}^{so(7,\mathbb{C})}(\lambda)$ is simple if $\lambda$ is anti-dominant by Proposition 5.1. Thus, suppose $\lambda=a\varepsilon_1+b\varepsilon_2+c\varepsilon_3$ for some $a,b,c\in\mathbb{C}$, an easy computation will show that $M_{\mathfrak{b}_{so(7,\mathbb{C})}}^{so(7,\mathbb{C})}(\lambda)$ is simple if $2a+4\notin\mathbb{N}$, $2b+2\notin\mathbb{N}$, $2c\notin\mathbb{N}$, $a-b\notin\mathbb{N}$, $b-c\notin\mathbb{N}$, $a-c+1\notin\mathbb{N}$, $a+b+3\notin\mathbb{N}$, $b+c+1\notin\mathbb{N}$ and $a+c+2\notin\mathbb{N}$. Under these conditions, any sub-quotient occurring in its restriction to $\mathfrak{g}_2$ lies in $\mathcal{O}^{\mathfrak{b}_{\mathfrak{g}_2}}$.
\begin{lemma}
Suppose that $M_{\mathfrak{b}_{so(7,\mathbb{C})}}^{so(7,\mathbb{C})}(\lambda)$ is simple with $\lambda=a\varepsilon_1+b\varepsilon_2+c\varepsilon_3$ for some $a,b,c\in\mathbb{C}$. If $a-b+2c\notin\mathbb{Z}$, $b-c\notin\mathbb{Z}$, $a+2b-c\notin\mathbb{Z}$, $2a+b+c\notin\mathbb{Z}$, $a+c\notin\mathbb{Z}$ and $a+b\notin\mathbb{Z}$, then
\begin{enumerate}[(1)]
\item each direct summand of the decomposition in Proposition 4.3 is simple as $\mathfrak{g}_2$-module;
\item any two direct summands of the decomposition in Proposition 4.3 have no non-split extensions.
\end{enumerate}
\end{lemma}
\begin{proof}
If $\delta=u\alpha_1+v\alpha_2$ for $u,v\in\mathbb{C}$ appears as a parameter of a direct summand of the decomposition in Proposition 4.3, then it satisfies that $a+c-u+v\in\mathbb{N}$ and $a+b-v\in\mathbb{N}$. Hence, $\frac{2(\delta+\rho(\mathfrak{g}_2),\alpha_1)}{(\alpha_1,\alpha_1)}=2u-3v+1\equiv a-b+2c(\mathrm{mod}\mathbb{Z})$. Thus if $a-b+2c\notin\mathbb{Z}$, $\frac{2(\delta+\rho(\mathfrak{g}_2),\alpha_1)}{(\alpha_1,\alpha_1)}\notin\mathbb{Z}$. By the similar process, one can easily check that $b-c\notin\mathbb{Z}$, $a+2b-c\notin\mathbb{Z}$, $2a+b+c\notin\mathbb{Z}$, $a+c\notin\mathbb{Z}$ and $a+b\notin\mathbb{Z}$ imply $\frac{2(\delta+\rho(\mathfrak{g}_2),\beta)}{(\beta,\beta)}\notin\mathbb{Z}$ for $\beta\in\Phi^+(\mathfrak{g}_2)-\{\alpha_1\}$. Therefore, $M_{\mathfrak{b}_{\mathfrak{g}_2}}^{\mathfrak{g}_2}(\delta)$ is simple as $\mathfrak{g}_2$-module by Proposition 5.1. This proves (1).

Now each parameter $\delta$ is anti-dominant in its linkage class, and it follows that different parameters of the Verma modules occurring in the decomposition in Proposition 4.3 lie in different linkage classes by Proposition 5.4. By Proposition 5.4 and 5.5, any two direct summands in the decomposition in Proposition 4.3 have no non-split extension in $\mathcal{O}^{\mathfrak{b}_{\mathfrak{g}_2}}$. Notice that $M_{\mathfrak{b}_{so(7,\mathbb{C})}}^{so(7,\mathbb{C})}(\lambda)$ is simple, any sub-quotient occurring in its restriction to $\mathfrak{g}_2$ lies in $\mathcal{O}^{\mathfrak{b}_{\mathfrak{g}_2}}$, so any two direct summands of the decomposition in Proposition 4.3 have no non-split extensions. (2) is proved.
\end{proof}
Notice that $b-c\notin\mathbb{Z}$, $a+c\notin\mathbb{Z}$ and $a+b\notin\mathbb{Z}$ imply $a-b\notin\mathbb{N}$, $a+c+2\notin\mathbb{N}$ and $a+b+3\notin\mathbb{N}$ respectively. If rearranging the twelve conditions above and define
\begin{center}
$S_{\mathfrak{b}_{so(7,\mathbb{C})}}=\{(r,s,t)\in\mathbb{C}^3\mid2r+4\notin\mathbb{N},2s+2\notin\mathbb{N},2t\notin\mathbb{N},r-s\notin\mathbb{N},
r-t+1\notin\mathbb{N},s+t+1\notin\mathbb{N},r+s\notin\mathbb{Z},r+t\notin\mathbb{Z},s-t\notin\mathbb{Z},r-s+2t\notin\mathbb{Z},r+2s-t\notin\mathbb{Z},
2r+s+t\notin\mathbb{Z}\}$
\end{center}
which is just
\begin{center}
$S_{\mathfrak{b}_{so(7,\mathbb{C})}}= \{\nu\in\mathfrak{h}_{so(7,\mathbb{C})}^*\mid\nu(H_{\varepsilon_1})+4\notin\mathbb{N},\nu(H_{\varepsilon_2})+2\notin\mathbb{N},
\nu(H_{\varepsilon_3})\notin\mathbb{N},\nu(H_{\varepsilon_1-\varepsilon_2})\notin\mathbb{N},\nu(H_{\varepsilon_1-\varepsilon_3})+1\notin\mathbb{N},
\nu(H_{\varepsilon_2+\varepsilon_3})+1\notin\mathbb{N},\nu(H_{\varepsilon_1+\varepsilon_2})\notin\mathbb{Z},
\nu(H_{\varepsilon_1+\varepsilon_3})\notin\mathbb{Z},\nu(H_{\varepsilon_2-\varepsilon_3})\notin\mathbb{Z},
\nu(H_{\varepsilon_1-\varepsilon_2}+H_{\varepsilon_3})\notin\mathbb{Z},\nu(H_{\varepsilon_1-\varepsilon_3}+H_{\varepsilon_2})\notin\mathbb{Z},
\nu(H_{\varepsilon_1}+H_{\varepsilon_2+\varepsilon_3})\notin\mathbb{Z}\}$,
\end{center}
then we obtain
\begin{theorem}
Let $\lambda\in S_{\mathfrak{b}_{so(7,\mathbb{C})}}$. Then the decomposition in Proposition 4.3 is a decomposition of simple $\mathfrak{g}_2$-modules.
\end{theorem}
\begin{proof}
By Proposition 4.3 and Lemma 5.6, the conclusion is proved immediately.
\end{proof}

\begin{center}
\textbf{5.3 $M_{\mathfrak{p}_{\{\varepsilon_2-\varepsilon_3\}}}^{so(7,\mathbb{C})}(\lambda)$}
\end{center}

Suppose $\lambda=a\varepsilon_1+b\varepsilon_2+c\varepsilon_3$ for some $a,b,c\in\mathbb{C}$ with $b-c\in\mathbb{N}$, by Proposition 5.1, an easy computation will show that $M_{\mathfrak{p}_{\{\varepsilon_2-\varepsilon_3\}}}^{so(7,\mathbb{C})}(\lambda)$ is simple if $2a+4\notin\mathbb{N}$, $2b+2\notin\mathbb{N}$, $2c\notin\mathbb{N}$, $a-b\notin\mathbb{N}$, $a-c+1\notin\mathbb{N}$, $a+b+3\notin\mathbb{N}$, $b+c+1\notin\mathbb{N}$ and $a+c+2\notin\mathbb{N}$.

Let $\delta=u\alpha_1+v\alpha_2$ for some $u,v\in\mathbb{C}$ with $2v-u\in\mathbb{N}$. Now we do some calculation for preparation. It is known that
\begin{center}
$W_{\mathfrak{g}_2}=\{\pm1,\pm s_{\alpha_1},\pm s_{\alpha_2},\pm s_{\alpha_2}s_{\alpha_1},\pm s_{\alpha_1}s_{\alpha_2},\pm s_{\alpha_1}s_{\alpha_2}s_{\alpha_1}\}$.
\end{center}
\begin{eqnarray*}
\begin{array}{rclcrcl}
1(\delta+\rho(\mathfrak{g}_2))-\rho(\mathfrak{g}_2)&=&u\alpha_1+v\alpha_2,\\
s_{\alpha_1}(\delta+\rho(\mathfrak{g}_2))-\rho(\mathfrak{g}_2)&=&(-u+3v-1)\alpha_1+v\alpha_2,\\
s_{\alpha_2}(\delta+\rho(\mathfrak{g}_2))-\rho(\mathfrak{g}_2)&=&u\alpha_1+(u-v-1)\alpha_2,\\
s_{\alpha_2}s_{\alpha_1}(\delta+\rho(\mathfrak{g}_2))-\rho(\mathfrak{g}_2)&=&(-u+3v-1)\alpha_1+(-u+2v-2)\alpha_2,\\
s_{\alpha_1}s_{\alpha_2}(\delta+\rho(\mathfrak{g}_2))-\rho(\mathfrak{g}_2)&=&(2u-3v-4)\alpha_1+(u-v-1)\alpha_2,\\
s_{\alpha_1}s_{\alpha_2}s_{\alpha_1}(\delta+\rho(\mathfrak{g}_2))-\rho(\mathfrak{g}_2)&=&(-2u+3v-6)\alpha_1+(-u+2v-2)\alpha_2,\\
-1(\delta+\rho(\mathfrak{g}_2))-\rho(\mathfrak{g}_2)&=&(-u-10)\alpha_1+(-v-6)\alpha_2,\\
-s_{\alpha_1}(\delta+\rho(\mathfrak{g}_2))-\rho(\mathfrak{g}_2)&=&(u-3v-9)\alpha_1+(-v-6)\alpha_2,\\
-s_{\alpha_2}(\delta+\rho(\mathfrak{g}_2))-\rho(\mathfrak{g}_2)&=&(-u-10)\alpha_1+(-u+v-5)\alpha_2,\\
-s_{\alpha_2}s_{\alpha_1}(\delta+\rho(\mathfrak{g}_2))-\rho(\mathfrak{g}_2)&=&(u-3v-9)\alpha_1+(u-2v-4)\alpha_2,\\
-s_{\alpha_1}s_{\alpha_2}(\delta+\rho(\mathfrak{g}_2))-\rho(\mathfrak{g}_2)&=&(-2u+3v-6)\alpha_1+(-u+v-5)\alpha_2,\\
-s_{\alpha_1}s_{\alpha_2}s_{\alpha_1}(\delta+\rho(\mathfrak{g}_2))-\rho(\mathfrak{g}_2)&=&(2u-3v-4)\alpha_1+(u-2v-4)\alpha_2.
\end{array}
\end{eqnarray*}
\begin{lemma}
Suppose that $M_{\mathfrak{p}_{\{\varepsilon_2-\varepsilon_3\}}}^{so(7,\mathbb{C})}(\lambda)$ is simple with $\lambda=a\varepsilon_1+b\varepsilon_2+c\varepsilon_3$ for $a,b,c\in\mathbb{C}$ satisfying $b-c\in\mathbb{N}$. If $a-b+2c\notin\mathbb{Z}$, $2a+b+c\notin\mathbb{Z}$, $a+2b-c\notin\mathbb{Z}$, $a+b\notin\mathbb{Z}$ and $a+c\notin\mathbb{Z}$, then
\begin{enumerate}[(1)]
\item each direct summand of the decomposition in Proposition 4.5 is simple as $\mathfrak{g}_2$-module;
\item any two direct summands of the decomposition in Proposition 4.5 have no non-split extensions.
\end{enumerate}
\end{lemma}
\begin{proof}
If $\delta=u\alpha_1+v\alpha_2$ for $u,v\in\mathbb{C}$ appears as a parameter of a direct summand of the decomposition in Proposition 4.5, then it satisfies that $2v-u\in\mathbb{N}$, $a+c-u+v\in\mathbb{N}$ and $a+b-v\in\mathbb{N}$. Hence, $\frac{2(\delta+\rho(\mathfrak{g}_2),\alpha_1)}{(\alpha_1,\alpha_1)}=2u-3v+1\equiv a-b+2c(\mathrm{mod}\mathbb{Z})$. Thus if $a-b+2c\notin\mathbb{Z}$, $\frac{2(\delta+\rho(\mathfrak{g}_2),\alpha_1)}{(\alpha_1,\alpha_1)}\notin\mathbb{Z}$. By the same calculation as in Lemma 5.6, $a-b+2c\notin\mathbb{Z}$, $a+2b-c\notin\mathbb{Z}$, $2a+b+c\notin\mathbb{Z}$, $a+c\notin\mathbb{Z}$ and $a+b\notin\mathbb{Z}$ imply $\frac{2(\delta+\rho(\mathfrak{g}_2),\beta)}{(\beta,\beta)}\notin\mathbb{Z}$ for $\beta\in\Phi^+(\mathfrak{g}_2)-\{\alpha_2\}$. Therefore, $M_{\mathfrak{p}_{\{\alpha_2\}}}^{\mathfrak{g}_2}(\delta)$ is simple as $\mathfrak{g}_2$-module by Proposition 5.1. This proves (1).

Again, we have $u\equiv 2a+b+c(\mathrm{mod}\mathbb{Z})$ and $v\equiv a+b(\mathrm{mod}\mathbb{Z})$. Because $a+2b-c\notin\mathbb{Z}$, $-u+3v-1\equiv a-b+2c\not\equiv 2a+b+c(\mathrm{mod}\mathbb{Z})$. According to our calculation above, $s_{\alpha_1}(\delta+\rho(\mathfrak{g}_2))-\rho(\mathfrak{g}_2)=(-u+3v-1)\alpha_1+v\alpha_2$ doesn't lie in the linkage class of $\delta$. By the similar process, $a-b+2c\notin\mathbb{Z}$, $2a+b+c\notin\mathbb{Z}$, $a+2b-c\notin\mathbb{Z}$, $a+b\notin\mathbb{Z}$ and $a+c\notin\mathbb{Z}$ will guarantee that $\omega(\delta+\rho(\mathfrak{g}_2))-\rho(\mathfrak{g}_2)$ for $\omega=W_{\mathfrak{g}_2}-\{s_{\alpha_2},-1\}$ don't lie in the linkage class of $\delta$. For $s_{\alpha_2}(\delta+\rho(\mathfrak{g}_2))-\rho(\mathfrak{g}_2)=u\alpha_1+(u-v-1)\alpha_2$, $2(u-v-1)-u=u-2v-2\notin\mathbb{N}$ since $2v-u\in\mathbb{N}$. Hence, it doesn't appear as the parameter of any direct summand of the decomposition in Proposition 4.5. Neither does $-1(\delta+\rho(\mathfrak{g}_2))-\rho(\mathfrak{g}_2)$ for the same reason. Thus, any two direct summands with different parameters in the decomposition in Proposition 4.5 have different infinitesimal characters. Notice that $M_{\mathfrak{p}_{\{\varepsilon_2-\varepsilon_3\}}}^{so(7,\mathbb{C})}(\lambda)$ is simple, any sub-quotient occurring in its restriction to $\mathfrak{g}_2$ lies in $\mathcal{O}^{\mathfrak{p}_{\{\alpha_2\}}}$, so any two direct summands of the decomposition in Proposition 4.5 have no non-split extensions by Remark 5.3, Proposition 5.4 and 5.5. (2) is proved.
\end{proof}
\begin{remark}
Because each direct summand of the decomposition in Proposition 4.5 isn't a standard Verma module, it is impossible for the parameter to be anti-dominant. Hence, the proof of the second part of Lemma 5.6 is not applicable here.
\end{remark}
Again, if we arrange all the required conditions in this section together with those appearing in Proposition 4.5, some overlapped conditions can be thrown away. Therefore, we obtain
\begin{theorem}
Let $\lambda\in\mathfrak{h}_{so(7,\mathbb{C})}^*$ satisfying $\lambda(H_{\varepsilon_2-\varepsilon_3})\in\mathbb{N}$, $\lambda(H_{\varepsilon_1})+4\notin\mathbb{N}$, $\lambda(H_{\varepsilon_2})+2\notin\mathbb{N}$ and $\lambda(H_{\varepsilon_1}+H_{\varepsilon_2+\varepsilon_3})\notin\mathbb{Z}$. Then the decomposition in Proposition 4.5 is a decomposition of simple $\mathfrak{g}_2$-modules.
\end{theorem}
\begin{proof}
By Proposition 4.5 and Lemma 5.8, the conclusion is proved immediately.
\end{proof}

\begin{center}
\textbf{5.4 $M_{\mathfrak{p}_{\{\varepsilon_1-\varepsilon_2,\varepsilon_3\}}}^{so(7,\mathbb{C})}(\lambda)$}
\end{center}

The method used in this section is parallel to that used in Section 5.3, so we just state the key results.

$M_{\mathfrak{p}_{\{\varepsilon_1-\varepsilon_2,\varepsilon_3\}}}^{so(7,\mathbb{C})}(\lambda)$ is simple if $2a+4\notin\mathbb{N}$, $2b+2\notin\mathbb{N}$, $b-c\notin\mathbb{N}$, $a-c+1\notin\mathbb{N}$, $a+b+3\notin\mathbb{N}$, $b+c+1\notin\mathbb{N}$ and $a+c+2\notin\mathbb{N}$.
\begin{lemma}
Suppose that $M_{\mathfrak{p}_{\{\varepsilon_1-\varepsilon_2,\varepsilon_3\}}}^{so(7,\mathbb{C})}(\lambda)$ is simple with $\lambda=a\varepsilon_1+b\varepsilon_2+c\varepsilon_3$ for $a,b,c\in\mathbb{C}$ satisfying $a-b,2c\in\mathbb{N}$. If $b-c\notin\mathbb{Z}$, $2a+b+c\notin\mathbb{Z}$, $a+2b-c\notin\mathbb{Z}$, $a+b\notin\mathbb{Z}$ and $a+c\notin\mathbb{Z}$, then
\begin{enumerate}[(1)]
\item each direct summand of the decomposition in Proposition 4.7 is simple as $\mathfrak{g}_2$-module;
\item any two direct summands of the decomposition in Proposition 4.7 have no non-split extensions.
\end{enumerate}
\end{lemma}
\begin{proof}
The proof is similar to that of Lemma 5.8.
\end{proof}
Rearrange all the conditions, and we obtain
\begin{theorem}
Let $\lambda\in\mathfrak{h}_{so(7,\mathbb{C})}^*$ satisfying $\lambda(H_{\varepsilon_1-\varepsilon_2})\in\mathbb{N}$, $\lambda(H_{\varepsilon_3})\in\mathbb{N}$, $\lambda(H_{\varepsilon_1+\varepsilon_2})\notin\mathbb{Z}$ and $\lambda(H_{\varepsilon_1}+H_{\varepsilon_2+\varepsilon_3})\notin\mathbb{Z}$. Then the decomposition in Proposition 4.7 is a  decomposition of simple $\mathfrak{g}_2$-modules.
\end{theorem}
\begin{proof}
By Proposition 4.7 and Lemma 5.11, the conclusion is proved immediately.
\end{proof}

\section{Branching Formulas for $(\mathfrak{g}_2,sl(3,\mathbb{C}))$}

We know that $\mathfrak{g}_2$ has four standard parabolic subalgebras, which are corresponding to $\phi$, $\{\alpha_1\}$, $\{\alpha_2\}$, and $\Delta(\mathfrak{g}_2)$.
\begin{proposition}
All of the four standard parabolic subalgebras of $\mathfrak{g}_2$ are $sl(3,\mathbb{C})$-compatible.
\end{proposition}
\begin{proof}
Because $\mathfrak{h}_{sl(3,\mathbb{C})}=\mathfrak{h}_{\mathfrak{g}_2}$, the conclusion is obvious.
\end{proof}
We shall not give details of computation for the three non-trivial cases because the computation is similar but much easier than that for $(so(7,\mathbb{C}),\mathfrak{g}_2)$. We only state the final results for them.

We give below the Chevalley basis of $sl(3,\mathbb{C})$ in terms of that of $\mathfrak{g}_2$.
\begin{eqnarray*}
\begin{array}{rclcrcl}
X_{\eta_1-\eta_2}&=&X_{\alpha_2}, & &  X_{\eta_2-\eta_1}&=&X_{-\alpha_2},\\
X_{\eta_2-\eta_3}&=&X_{3\alpha_1+\alpha_2}, & & X_{\eta_3-\eta_2}&=&X_{-3\alpha_1-\alpha_2},\\
X_{\eta_1-\eta_3}&=&X_{3\alpha_1+2\alpha_2}, & & X_{\eta_3-\eta_1}&=&X_{-3\alpha_1-2\alpha_2},\\
H_{\eta_1-\eta_2}&=&H_{\alpha_2},\\
H_{\eta_2-\eta_3}&=&H_{3\alpha_1+\alpha_2}.
\end{array}
\end{eqnarray*}
A Cartan subalgebra $\mathfrak{h}_{sl(3,\mathbb{C})}$ of $sl(3,\mathbb{C})$ is complex linearly spanned by $\{H_{\eta_1-\eta_2},\\H_{\eta_2-\eta_3}\}$. Fix simple roots $\Delta(sl(3,\mathbb{C}))=\{\eta_1-\eta_2,\eta_2-\eta_3\}$ of $sl(3,\mathbb{C})$. Because $\eta_1+\eta_2+\eta_3=0$ on $\mathfrak{h}_{sl(3,\mathbb{C})}^*$, every element $\delta\in\mathfrak{h}_{sl(3,\mathbb{C})}^*$ can be uniquely written as $x\eta_1+y\eta_2$ for some $x,y\in\mathbb{C}$.

Moreover, it is not hard to compute that $\mathfrak{b}_{\mathfrak{\mathfrak{g}_2}}\cap sl(3,\mathbb{C})=\mathfrak{p}_{\{\alpha_1\}}\cap sl(3,\mathbb{C})=\mathfrak{b}_{sl(3,\mathbb{C})}$ and $\mathfrak{p}_{\{\alpha_2\}}\cap sl(3,\mathbb{C})=\mathfrak{p}_{\{\eta_1-\eta_2\}}$.

Let $S_{\mathfrak{b}_{\mathfrak{g}_2}}=\{\nu\in\mathfrak{h}_{\mathfrak{g}_2}^*\mid\nu(H_{\alpha_1})\notin\mathbb{N},\nu(H_{\alpha_1+\alpha_2})+3\notin\mathbb{N},
\nu(H_{2\alpha_1+\alpha_2})+4\notin\mathbb{N},\nu(H_{\alpha_2})\notin\mathbb{Z},\nu(H_{3\alpha_1+\alpha_2})\notin\mathbb{Z},
\nu(H_{3\alpha_1+2\alpha_2})\notin\mathbb{Z}\}$.
\begin{theorem}
Let $\lambda\in\mathfrak{h}_{\mathfrak{g}_2}^*$ and $\delta\in\mathfrak{h}_{sl(3,\mathbb{C})}^*$. Denote $\sigma=\mathrm{Res}_{sl(3,\mathbb{C})}^{\mathfrak{g}_2}\lambda+\delta$ and $\mu=\mathrm{Res}_{sl(3,\mathbb{C})}^{\mathfrak{g}_2}\lambda-\delta$.
\begin{enumerate}
\item If $\lambda\in S_{\mathfrak{b}_{\mathfrak{g}_2}}$, then
    \begin{center}
    $\mathrm{Res}_{sl(3,\mathbb{C})}^{\mathfrak{g}_2}M_{\mathfrak{b}_{\mathfrak{g}_2}}^{\mathfrak{g}_2}(\lambda)=
    \displaystyle{\bigoplus_{\mbox{\tiny$\begin{array}{c}\mu(H_{\eta_1-\eta_3})\in\mathbb{N}\\ \mu(H_{\eta_2-\eta_3})\in\mathbb{N}\end{array}$}}}(1+\min\{\mu(H_{\eta_1-\eta_3}),\mu(H_{\eta_2-\eta_3})\})
    M_{\mathfrak{b}_{sl(3,\mathbb{C})}}^{sl(3,\mathbb{C})}(\delta)$
    \end{center}
    is a decomposition of simple $sl(3,\mathbb{C})$-modules.
\item If $\lambda(H_{\alpha_1})\in\mathbb{N}$, $\lambda(H_{2\alpha_1+\alpha_2})+4\notin\mathbb{N}$ and $\lambda(H_{3\alpha_1+2\alpha_2})\notin\mathbb{Z}$, then
    \begin{center}
    $\mathrm{Res}_{sl(3,\mathbb{C})}^{\mathfrak{g}_2}M_{\mathfrak{p}_{\{\alpha_1\}}}^{\mathfrak{g}_2}(\lambda)
    =\displaystyle{\bigoplus_{\mbox{\tiny$\begin{array}{c}\mu(H_{\eta_1-\eta_3})\in\mathbb{N}\\ \mu(H_{\eta_2-\eta_3})\in\mathbb{N}\\ \delta(H_{\eta_1-\eta_2})\leq\lambda(H_{3\alpha_1+2\alpha_2})\end{array}$}}}(1+X-Y)M_{\mathfrak{b}_{sl(3,\mathbb{C})}}^{sl(3,\mathbb{C})}(\delta)$
    \end{center}
    is a decomposition of simple $sl(3,\mathbb{C})$-modules, where
    \begin{center}
    $X=\min\{\mu(H_{\eta_1-\eta_3}),\mu(H_{\eta_2-\eta_3})\}$,\\
    $Y=\max\{\mu(H_{\eta_2-\eta_3})-\lambda(H_{\alpha_1}),0\}$.
    \end{center}
\item If $\lambda(H_{\alpha_2})\in\mathbb{N}$ and $\lambda(H_{2\alpha_1+\alpha_2})\notin\mathbb{Z}$, then
    \begin{center}
    $\mathrm{Res}_{sl(3,\mathbb{C})}^{\mathfrak{g}_2}M_{\mathfrak{p}_{\{\alpha_2\}}}^{\mathfrak{g}_2}(\lambda)
    =\displaystyle{\bigoplus_{\mbox{\tiny$\begin{array}{c}\delta(H_{\eta_1-\eta_2})\in\mathbb{N}\\ \mu(H_{\eta_1-\eta_3})\in\mathbb{N}\\ \mu(H_{\eta_2-\eta_3})\in\mathbb{N}\end{array}$}}}(1+\frac{Z-|\mu(H_{\eta_1-\eta_2})|}{2})M_{\mathfrak{p}_{\{\eta_1-\eta_2\}}}^{sl(3,\mathbb{C})}(\delta)$
    \end{center}
    is a decomposition of simple $sl(3,\mathbb{C})$-modules, where $Z=\min\{\mu(H_{\eta_1-\eta_3}+H_{\eta_2-\eta_3}),\sigma(H_{\eta_1-\eta_2})\}$.
\end{enumerate}
\end{theorem}
\begin{remark}
In each decomposition of Theorem 5.7, 5.9, 5.12 and 6.2, all the highest weight vectors of direct summands are $\bar{\mathfrak{b}}$-singular vectors defined in $[\textbf{MS1}]$. In that paper, the authors listed the $\bar{\mathfrak{b}}$-singular vectors of $V_\lambda(\mathfrak{l})$ which is our $F_\lambda$, for $\lambda$ ``small''. In fact, one can check that those vectors are contained in our results. On the other hand, a method called F-method is introduced in $[\textbf{MS2}]$. In that paper, the authors used F-method to find out the space of $\tilde{L}'$-singular vectors (Definition 3.1 [\textbf{MS2}]). Although the space of $\tilde{L}'$-singular vectors contain more elements than the set of $\bar{\mathfrak{b}}$-singular vectors, i.e., some of them may not be useful to branching formulas, the F-method offers a new tool to study branching laws, at least shrinks the range of highest weight vectors.
\end{remark}

\end{document}